\documentclass[10pt,draft]{amsart}
\usepackage[T1]{fontenc} \usepackage{mathptmx} 
\usepackage{amssymb}

\newtheorem{theorem}{Theorem}
\newtheorem{prop}[theorem]{Proposition}

\newcommand{\del}[1]{\frac{\partial}{\partial #1}}
\newcommand{\indel}[1]{\partial/\partial #1}

\begin{document}

\title[Locally homogeneous affine connections on surfaces]{A classification of locally homogeneous affine connections on compact surfaces}

\author{Adolfo Guillot}
\author{Antonia S\'anchez Godinez}
\address{Unidad Cuernavaca, Instituto de Matem\'{a}ticas, Universidad Nacional Aut\'{o}noma de M\'{e}xico, Av. Universidad s/n, Lomas de Chamilpa, C.P.~62210, Cuernavaca, Morelos, M\'exico.}
\email{adolfo@matcuer.unam.mx, antonia@matcuer.unam.mx}
\thanks{Partially supported by PAPIT-UNAM IN108214 (Mexico)}

\date{\today}

\subjclass[2010]{Primary 53C30, 53A15, 51H25}

\begin{abstract}We classify the affine connections on compact orientable surfaces for which the pseudogroup of local isometries acts transitively. We prove that such a connection is either torsion-free and flat, the Levi-Civita connection of a Riemannian metric of constant curvature or the quotient of a translation-invariant connection in the plane. 
\end{abstract}

\maketitle

\section{Introduction}

The most symmetric Riemannian metrics on compact orientable surfaces are those for which the isometry pseudogroup acts transitively. In such a case, the Gaussian curvature is necessarily constant and, thus, the surface is locally modeled on a simply connected surface endowed with a complete metric of constant curvature (the sphere, the Euclidean plane or the hyperbolic one, up to similitudes). Moreover, the surface is a global quotient of this space. The aim of this article is to classify, in the same spirit, the most symmetric   affine connections on surfaces:

\begin{theorem}\label{thmmain} Let~$S$ be a compact orientable surface endowed with a~$C^\infty$   affine connection~$\nabla$ such that the pseudogroup of local isometries of~$\nabla$ acts transitively upon~$S$. Either
\begin{itemize}
 \item $\nabla$ is torsion-free and flat, 
 \item $\nabla$ is the Levi-Civita connection of a Riemannian metric on~$S$, or
 \item $S$ is the quotient of~$\mathbf{R}^2$ under the action of a lattice and $\nabla$ is induced by a translation-invariant connection on~$\mathbf{R}^2$.
\end{itemize} \end{theorem}

Opozda proved that a compact orientable surface admitting a locally homogeneous connection which is not the Levi-Civita connection of a Riemannian metric is a torus~\cite{opozda}. Our result can be seen as a refinement of this statement: such a connection is either torsion-free and flat (and thus fits within the description given by Nagano and Yagi~\cite{naganoyagi}) or the global quotient (by translations) of a translation-invariant connection in the plane. 

Our discussion splits naturally into a local and a global part. The local one is based upon Lie's classification of transitive transformation groups in dimension two. For our local analysis, we refine, in Theorem~\ref{liecorr}, some aspects of the local classification of locally homogeneous connections and their transformation groups, studied in~\cite{ATO} and \cite{kov}. The global part uses the~$(G,X)$\nobreakdash-structure  (in the sense of Thurston, \cite[\S 3.4]{thurston}) associated to some locally homogeneous rigid geometric structure (a connection in the present case), as detailed in~\cite{dum-gui}.

We heartily thank Sorin Dumitrescu for many interesting conversations and for his comments on this text.

\section{Connections and their local symmetries}

For a manifold~$M$, we denote by~$\mathfrak{X}(M)$ the sheaf of its~$C^\infty$ vector fields. An \emph{(affine) connection} on~$M$ is a map~$\nabla:\mathfrak{X}(M)\times \mathfrak{X}(M)\to \mathfrak{X}(M)$, $\nabla_X Y=\nabla(X,Y)$,  such that, for functions~$f$ and~$g$, $\nabla_{fX}gY=fg\nabla_XY+f(X\cdot g)Y$. To a connection~$\nabla$ we may associate its \emph{torsion} tensor~$T(X,Y)=\nabla_X Y-\nabla_YX-[X,Y]$ and its~\emph{curvature} tensor~$R(X,Y) =\nabla_X\nabla_Y-\nabla_Y\nabla_X-\nabla_{[X,Y]}$. A connection with vanishing torsion is called~\emph{torsion-free} or \emph{symmetric};  one with vanishing curvature, \emph{flat}. A torsion-free and flat connection is, in suitable coordinates, the standard affine connection in~$\mathbf{R}^n$, the unique connection such that~$\nabla_XY=0$ for every pair~$(X,Y)$ of constant vector fields (see~\cite[\S 1.7]{wolf} or \cite[Ch.~I, \S 10]{eisenhart}).

For~$U\subset M$, a mapping~$f:U\to M$ is a \emph{local isometry} if it preserves~$\nabla$. A vector field~$X$ in~$U$ is said to be a \emph{Killing field} of~$\nabla$ is its local flow is composed of local isometries. Local isometries form naturally a Lie pseudogroup whose Lie algebra is the one formed by Killing vector fields. In particular, if the isometry pseudogroup of~$\nabla$ acts transitively in a neighborhood of~$p\in M$, there exist two Killing vector fields of~$\nabla$ in a neighborhood of~$p$ which are linearly independent. A vector field~$X$ is a Killing field of the connection~$\nabla$ if for every pair of (local) vector fields~$Y$ and~$Z$ we have
\begin{equation}\label{equationkilling}[X,\nabla_Y Z]-\nabla_Y[X,Z]=\nabla_{[X,Y]}Z.\end{equation}
(See~\cite[Ch.~VI, Prop.~2.2]{kobayashi-nomizu}).

In coordinates, in the neighborhood of a point in~$\mathbf{R}^2$, the most general connection has the form
\begin{multline}\label{connection}\nabla_{\del{x}}\del{x}=A\del{x}+B\del{y},\; \nabla_{\del{x}}\del{y}=\left(C+{\textstyle\frac{1}{2}}U\right)\del{x}+\left(D+{\textstyle\frac{1}{2}}V\right)\del{y}, \\ \nabla_{\del{y}}\del{x}=\left(C-{\textstyle\frac{1}{2}}U\right)\del{x}+\left(D-{\textstyle\frac{1}{2}}V\right)\del{y},\;\nabla_{\del{y}}\del{y}=E\del{x}+F\del{y},\end{multline}
where~$A$, $B$, $C$, etc. are functions of~$x$ and~$y$. The torsion-free case corresponds to~$U \equiv 0$ and~$V\equiv 0$. The vector field~$X=a(x,y)\indel{x}+b(x,y)\indel{y}$ will be a vector field of the above connection if and only if its coefficients satisfy the system of  partial differential equations
\begin{eqnarray}
\label{eq1} 0 & = & a_{xx}+Aa_x-Ba_y+2Cb_x+A_x a+A_y b,\\
   0 & = & b_{xx}+2Ba_x+(2D-A)b_x-Bb_y+B_x a+B_y b,\\
    0 & = & a_{xy}+(A-D)a_y+Eb_x+Cb_y+C_xa+C_y b,\\
   0 & = & b_{xy}+Da_x+Ba_y+(F-C)b_x+D_x a+D_y b,\\
   0 & = & a_{yy}-Ea_x+(2C-F)a_y+2Eb_y+E_x a+E_y b,\\
\label{eq6} 0 & = & b_{yy}+2Da_y-Eb_x+Fb_y+F_x a+F_y b;\end{eqnarray}
\begin{eqnarray} 0 & = & aU_x+bU_y-a_yV+b_yU,\\
\label{eqfin} 0 & = &  aV_x+bV_y+a_x V-b_xU.
\end{eqnarray}
These equations are obtained through the expressions of~(\ref{equationkilling}) when~$(Y,Z)$ takes values on all ordered couples of  coordinate vector fields (another formulation is given in~\cite{ATO}; in the torsion-free case the equations reduce to those considered in~\cite{kov}). The last two equations express the fact that~$X$ preserves the torsion tensor of~$\nabla$.\\

Our starting point is the following result, which summarizes  Lie's classification of two-dimen\-sional transformation groups in the presence of an invariant connection:

\begin{prop}\label{baselie} Let~$\nabla$ be a~$C^\infty$  affine connection in a neighborhood of the origin of~$\mathbf{R}^2$. If there are two Killing vector fields that are linearly independent at~$0$ and the connection is not torsion-free and flat, then the Lie algebra of Killing vector fields satisfies at least one of the following:
\begin{enumerate}
\item \label{2d} is two-dimensional;
\item \label{clas-so3} is isomorphic to~$\mathfrak{o}(3,\mathbf{R})$
\item  \label{clas-sl2} is isomorphic to~$\mathfrak{sl}(2,\mathbf{R})$
\item \label{mask2} in suitable coordinates, contains the vector fields~$\indel{y}$ and those of the form $h(y)\indel{x}$ for the solutions~$h(y)$ of a second-order linear differential equation with constant coefficients.
\item \label{mask3} in suitable coordinates, contains the vector fields~$\indel{y}$ and $x\indel{x}$. 
\item \label{rabtor}  in suitable coordinates, contains the vector fields~$\indel{y}$ and $\indel{x}$ and $(sx+y)\indel{x}+(sy-x)\indel{y}$ for some~$s\in\mathbf{R}$. 
\end{enumerate}
 \end{prop}

This statement comes from investigating, for every one of the transitive Lie transformation groups in~$\mathbf{R}^2$ (as classified by Lie in the \emph{Gruppenregister} and presented by Olver in~\cite{olver}), the affine connections that are preserved. Tables~1 and~6 in pages 472 and 475 of~\cite{olver} present these groups, labeled 1.1 through~1.11 (for the imprimitive ones) and~6.1 through~6.8 (for the primitive ones). The   groups 1.5--1.11 contain a parameter~$k\in \mathbf{Z}$ ($k\geq 1$); the groups~6.1 and~1.7 contain a parameter~$\alpha\in\mathbf{R}$. 

Most of these transformation groups do not preserve a connection. Indeed, if~$\nabla$ is a   connection in the neighborhood of the origin of~$\mathbf{R}^n$ and if~$X$ is a Killing vector field of~$\nabla$ vanishing at~$0$, $X$ is linear in exponential coordinates~\cite[Thm.~1.6.20]{wolf}. There is thus no nonzero Killing vector field with vanishing linear part. This prevents, a priori, cases 1.4, 1.5--1.9 for $k\geq 3$, 1.10--1.11 (for all values of~$k$), 6.7 and 6.8 from preserving a connection.

Many groups fall within the scope of item~(\ref{mask2}). This is the
case for the groups 1.5 and 1.6 (both in the case $k=2$). A particular
case of item~(\ref{mask2}) corresponds to~$h''=0$. It covers the
groups 1.7, 1.8, 1.9 (these three for all values of~$k$), 6.5 and 6.6. The group 1.5 ($k=1$) belongs to item~(\ref{2d}). The group 1.6 ($k=1$) corresponds to~(\ref{mask3}). The groups~1.1, 1.2, and 6.2 are covered by~(\ref{clas-sl2}). The group~6.3 belongs to~(\ref{clas-so3}). The group~1.3 is defined in~$\{(x,u)|u>0\}$ and contains the vector fields~$x\indel{x}, u\indel{u}$: in the coordinates~$(x,y)=(x,\log u)$, it belongs to case~(\ref{mask3}). The groups 6.1 and 6.4 belong to (\ref{rabtor}).

All these facts add up to Proposition~\ref{baselie}.\\

A complete picture should explicit the cases where the Lie transformation group preserving the connection is the maximal one. To this respect, we may refine the above statement in the following way (compare with~\cite{kov} and~\cite{ATO}):

\begin{theorem}\label{liecorr} Let~$\nabla$ be a~$C^\infty$   affine connection in a neighborhood of the origin of~$\mathbf{R}^2$. If there are two Killing vector fields that are linearly independent at~$0$ and the connection is not torsion-free and flat, then the Lie algebra of Killing vector fields has dimension at most four and is either:
\begin{enumerate}
\item two-dimensional;
\item  isomorphic to~$\mathfrak{o}(3,\mathbf{R})$;
\item   isomorphic to~$\mathfrak{sl}(2,\mathbf{R})$; or
\item\label{itemcorr} in suitable coordinates, generated by the vector fields~$x\indel{x}$, $\indel{y}$ and those of the form $h(y)\indel{x}$ for the solutions~$h(y)$ of a second-order linear differential equation  with constant coefficients.
\end{enumerate}
 \end{theorem}

\begin{proof} We will prove that in cases~(\ref{mask2}) and (\ref{mask3}) of Proposition~\ref{baselie}, we are actually in case~(\ref{itemcorr}) of Theorem~\ref{liecorr} and that in case (\ref{rabtor}) we have a torsion-free and flat connection.

We begin with case~(\ref{mask2}) of Proposition~\ref{baselie}. Let~$\alpha,\beta\in\mathbf{R}$ and let~$\nabla$ be a connection in a neighborhood of the origin of~$\mathbf{R}^2$. Suppose that the Killing Lie algebra of~$\nabla$ contains the three-dimensional Lie algebra of vector fields~$\mathfrak{K}_{\alpha,\beta}$ generated by~$\indel{y}$ and by the vector fields of the form~$h(y)\indel{x}$ such that
\begin{equation}\label{dekillingk}h''+\alpha h'+\beta h=0.\end{equation}
Since~$\indel{y}$ is a Killing vector field, $\nabla$ is of the form~(\ref{connection}) and all the coefficients are functions of~$x$ (independent of~$y$). If~$h(y)\indel{x}$ is a Killing vector field of~$\nabla$, it should satisfy the equations~(\ref{eq1})--(\ref{eqfin}), which read
\begin{multline}\label{syschris} -B h' +A' h=0,\; B' h =0,\; (A-D) h' +C' h =0,\;  B h' +D' h =0, \\ h'' +(2C-F) h' +E' h =0,\;2D h' +F' h =0,\; Vh'-U'h=0,\; V'h=0.
 \end{multline}
All these are linear differential equations in~$h(y)$ whose coefficients are functions of~$x$. Since, by hypothesis, every solution of~(\ref{dekillingk}) should be a solution to all the equations in the system~(\ref{syschris}), all the  first-order equations of the system should be identically zero and the second-order one should be~(\ref{dekillingk}). We conclude that
$$A=0,\; B=0,\; D=0,\; C=\gamma,\; F=2\gamma-\alpha,\; E=\beta x+\delta, U=2\upsilon,\; V=0,$$
for some~$\delta,\gamma,\upsilon\in\mathbf{R}$. The connection reduces to
\begin{multline}\label{casogen}\nabla_{\del{x}}\del{x}=0,\; \nabla_{\del{x}}\del{y}=(\gamma+\upsilon)\del{x},\; \nabla_{\del{y}}\del{x}=(\gamma-\upsilon)\del{x},\\ \nabla_{\del{y}}\del{y}=(\beta x+\delta)\del{x}+(2\gamma-\alpha)\del{y}.\end{multline}
The parameter~$\delta$ is not essential. If~$q$ is a function such that $q''+\alpha q'+\beta q+\delta=0$, the change of coordinates~$(x,y)\mapsto(x+q(y),y)$ preserves the Lie algebra~$\mathfrak{K}_{\alpha,\beta}$ and maps the connection~(\ref{casogen}) to one with the same values of~$\alpha$, $\beta$, $\gamma$ and~$\upsilon$ but with~$\delta=0$. Accordingly, we will henceforth assume that~$\delta=0$. For the torsion and curvature of~$\nabla$ we have
\begin{equation}\label{tor}\nabla_{\del{x}}\del{y}-\nabla_{\del{y}}\del{x}=2\upsilon,\end{equation}
\begin{equation}\label{curv4d}\left(\nabla_{\del{x}}\nabla_{\del{y}}-\nabla_{\del{y}}\nabla_{\del{x}}\right)\del{x}=0, \; \left(\nabla_{\del{x}}\nabla_{\del{y}}-\nabla_{\del{y}}\nabla_{\del{x}}\right)\del{y}=[\beta+(\gamma+\upsilon)(\gamma+\upsilon-\alpha)]\del{x}. \end{equation}
Let us now calculate the full Lie algebra of Killing vector fields of~$\nabla$ in the case where it is not torsion-free and flat. The system~(\ref{eq1})--(\ref{eqfin}) reduces to \begin{eqnarray}
\label{primera} 0 & = & a_{xx}+2\gamma b_x ,\\
\label{segunda}  0 & = &   b_{xx},\\
\label{tercera}  0 & = &  a_{xy}+\beta xb_x+\gamma b_y   , \\
\label{cuarta}   0 & = &  b_{xy}+(\gamma-\alpha)b_x  , \\
\label{quinta}  0 & = &  a_{yy}-\beta xa_x+\alpha a_y+2\beta x b_y+\beta a ,\\
\label{sexta}  0 & = &  b_{yy}-\beta xb_x+(2\gamma-\alpha)b_y,\end{eqnarray}
plus
\begin{equation}\label{sieteyocho} \upsilon b_y =0,\; \upsilon b_x=0.\end{equation}
Thus, $a(x,y)\indel{x}+b(x,y)\indel{y}$ will be a Killing vector field of the above connection if and only if it satisfies this system. Our discussion splits naturally in two cases.
\begin{description}
\item[Torsion-free] If~$\upsilon=0$, (\ref{sieteyocho}) is automatically satisfied. From~(\ref{segunda}) we obtain that~$b=f(y)x+g(y)$. Substituting this into~(\ref{cuarta}) and~(\ref{sexta}) we obtain, respectively, 
\begin{equation}\label{temp1}f'=(\alpha-\gamma)f,\end{equation}
\begin{equation}\label{temp2}x[f''+(2\gamma-\alpha)f'-\beta f]+[g''+(2\gamma-\alpha)g']=0.\end{equation}
The second equation splits as two separate ones, since the terms in brackets are functions of~$y$. Substituting~(\ref{temp1}) in~(\ref{temp2}), we have~$(\beta+\gamma^2-\gamma\alpha)f=0$. Since we are assuming that the connection is not flat, $\beta+\gamma^2-\gamma\alpha\neq 0$ and thus $f=0$. From equation~(\ref{primera}) we have that~$a=k(y)x+h(y)$. Equation~(\ref{tercera}) becomes~$k'+\gamma g'=0$. After solving~$k'$ from the latter equation and~$g''$ from~(\ref{temp2}) in terms of~$g'$, equation~(\ref{quinta}) becomes 
$$2x(\beta+\gamma^2-\gamma\alpha)g' +[h''+\alpha h'+\beta h] =0$$
The system of equations reduces to $h''+\alpha h'+\beta h=0$, $g'=0$ and $k'=0$: we are in case~(\ref{itemcorr}) of Theorem~\ref{liecorr}.
\item[Not torsion-free] If~$\upsilon\neq 0$, from~(\ref{sieteyocho}), $b$ is constant and the system~(\ref{eq1})--(\ref{eq6}) reduces to $ a_{xx}=0$, $ a_{xy}=0$, $ a_{yy}-\beta xa_x+\alpha a_y+\beta a=0$. We conclude that~$a=k(y)x+h(y)$, that $k$ is constant and that~$h$ satisfies the differential equation~(\ref{dekillingk}). We are still in case~(\ref{itemcorr}) of Theorem~\ref{liecorr}.
\end{description}

In case~(\ref{mask3}) of Proposition~\ref{baselie}, since~$\indel{y}$ is a Killing vector field, the coefficients of the connection are functions of~$x$. Since~$x\indel{x}$ is, by hypothesis, a Killing vector field, the coefficients of the connection must satisfy the system $xA'+A=0$, $xB'+2B=0$, $xC'=0$, $xD'+D=0$, $xE'-E=0$, $xF'=0$, $xU'=0$, $xV'+V=0$. We conclude that~$A$, $B$, $D$ and~$V$ vanish, that~$C$, $F$ and~$U$ are constants and that~$E=\mu x$ for some~$\mu\in\mathbf{R}$.  The connection belongs to the family~(\ref{casogen}) for~$\beta=\mu$ and~$\delta=0$. By the previous study of the Killing fields of the connection~(\ref{casogen}), we are in case~(\ref{itemcorr}) of Theorem~\ref{liecorr}.

In case~(\ref{rabtor}) of Proposition~\ref{baselie}, if a connection~$\nabla$ has such a Lie algebra of Killing vector fields then, since the coordinate vector fields are Killing ones, the Christoffel symbols are constants.  By imposing the condition that~$(sx+y)\indel{x}+(sy-x)\indel{y}$ is a Killing vector field, the system~(\ref{primera})--(\ref{sieteyocho})  reduces to a linear system of eight equations in the eight constants $A,B,C$, etc. whose coefficients are polynomials in~$s$. The determinant of the system is~$(s^2+9)(s^2+1)^3$. The only connection that is preserved is the one with vanishing Christoffel symbols.\end{proof}

\section{Proof of the Main Theorem}
Theorem~\ref{liecorr} classifies the germs of locally homogeneous connexions.  In order to globalize this result into Theorem~\ref{thmmain}, we will resort to the theory of $(G,X)$\nobreakdash-structures, which we briefly recall (see~\cite{thurston} for details). If~$G$ is a Lie group acting transitively on the simply connected manifold~$X$, a~\emph{$(G,X)$\nobreakdash-structure} on a manifold~$M$ is an atlas for its smooth structure taking values on~$X$ whose changes of coordinates are restrictions of elements of~$G$. To a~$(G,X)$\nobreakdash-structure on~$M$ there corresponds a \emph{developing map} $\mathcal{D}:\widetilde{M}\to X$ and a \emph{holonomy} morphism~$\rho:\pi_1(M)\to G$ that satisfy, for every~$\nu\in\pi_1(M)$, the relation
\begin{equation}\label{dev-hol}\mathcal{D}(\nu\cdot p)=\rho(\nu)\cdot \mathcal{D}(p).\end{equation}

We will use~$(G,X)$\nobreakdash-structures through the following result~\cite{dum-gui}:
 
\begin{prop}\label{constgx} Let~$M$ be an orientable manifold, $\nabla$ an affine connection on~$M$. Suppose that~$\nabla$ is locally homogeneous everywhere. Let~$p\in M$, let~$\mathfrak{g}$ be the Lie algebra of Killing vector fields of~$\nabla$ in a neighborhood of~$p$ and let~$\mathfrak{g}_0\subset \mathfrak{g}$ the subalgebra of those Killing vector fields vanishing at~$p$. Let~$G$ be the connected and simply connected Lie group corresponding to~$\mathfrak{g}$ and~$G_0\subset G$ the subgroup corresponding to~$\mathfrak{g}_0$. If~$G_0$ is closed, then there exists a finite Galois covering~$\overline{\pi}:\overline{M}\to M$ with Galois group~$\mathrm{Isom}^+(\nabla_0)/G$, a connection~$\nabla_0$ on~$G/G_0$ and a~$(G,G/G_0)$\nobreakdash-structure on~$\overline{M}$ that is, moreover, an isometry between~$\overline{\pi}^*\nabla$ and~$\nabla_0$. \end{prop}

The covering~$\overline{M}$ is not necessarily connected: it may have a connected component in restriction to which~$\overline{\pi}$ is one-to-one.\\

We will study all the local models of Theorem~\ref{liecorr} in order to classify the compact models of each one of those geometries.

\subsection{The four-dimensional Lie algebras}
As we saw in the proof, cases~(\ref{itemcorr}) of Theorem~\ref{liecorr} correspond to the connections
\begin{equation}\label{conn4d} \nabla_{\del{x}}\del{x}=0,\; \nabla_{\del{x}}\del{y}=(\gamma+\upsilon)\del{x},\; \nabla_{\del{y}}\del{x}=(\gamma-\upsilon)\del{x},\;  \nabla_{\del{y}}\del{y}=\beta x\del{x}+(2\gamma-\alpha)\del{y} \end{equation}
for $\alpha$, $\beta$, $\gamma$ and~$\upsilon$ constants, but such that the torsion~(\ref{tor}) and the curvature~(\ref{curv4d}) do not vanish simultaneously.  These  connections  form  actually a family of dimension three, since the rescaling~$y\mapsto \mu y$ preserves the family and acts upon the coefficients as \begin{equation}\label{rescaling}(\alpha,\beta,\gamma,\upsilon)\mapsto (\mu\alpha,\mu^2\beta,\mu\gamma,\mu\upsilon).\end{equation}

\subsubsection{The isometries} We begin the study of the connections~(\ref{conn4d}) by studying their isometries. Let~$\mathcal{M}$ denote~$\mathbf{R}^2$ with the connection~$\nabla_0$ given by~(\ref{conn4d}). Notice that~$\kappa:\mathcal{M}\to \mathcal{M}$ given by~$\kappa(x,y)=(-x,y)$ is a globally-defined orientation-reversing isometry. Let~$\mathfrak{g}$ be the Lie algebra of vector fields on~$\mathcal{M}$ corresponding to the Killing Lie algebra of~$\nabla_0$ (it depends upon~$\alpha$ and~$\beta$ but not upon~$\gamma$ or~$\upsilon$). Let~$G$ be the connected and simply connected Lie group associated to~$\mathfrak{g}$.

We will begin by proving that~$\mathfrak{g}$ is induced by a transitive action of~$G$ and thus that~$\mathcal{M}$ is a homogeneous space of~$G$ (in particular, every Killing vector field is complete). Denote by~$\mathrm{Isom}^+(\nabla_0)$ the group of orientation-preserving isometries of~$\nabla_0$. We naturally have~$G\vartriangleleft\mathrm{Isom}^+(\nabla_0)$.

The Lie algebra~$\mathfrak{g}$ contains the subalgebra~$\mathfrak{h}\subset \mathfrak{g}$ generated by the commuting vector fields $X=x\indel{x}$ and $Y=\indel{y}$. These vector fields are complete and induce an effective action of~$\mathbf{R}^2$ on~$\mathcal{M}$ given by
\begin{equation}\label{actionh}(s,t)\cdot(x,y)= (xe^s,y+t).\end{equation} 
The Lie algebra~$\mathfrak{g}$ contains also the commutative two-dimensional subalgebra~$\mathfrak{k}\subset \mathfrak{g}$ of the vector fields of the form~$h(y)\indel{x}$ for the solutions~$h(y)$ of~(\ref{dekillingk}). Each of these vector fields is complete and its flow on~$M$ in time~$\tau$ is given by
\begin{equation}\label{actionk}(u,v)\cdot (x,y)=(x+\tau h(y),y).\end{equation}

As a vector space, $\mathfrak{g}=\mathfrak{h}\oplus\mathfrak{k}$. We have~$[X,h \indel{x}]=-h \indel{x}$ and~$[Y,h \indel{x}]=h' \indel{x}$ ($h'$ is a solution of~(\ref{dekillingk}) if~$h$ is). Hence, $\mathfrak{k}$ is the derived Lie algebra of~$\mathfrak{g}$, this is, $[\mathfrak{g},\mathfrak{g}]=\mathfrak{k}$.

By the completeness of each of the vector fields generating~$\mathfrak{g}$, the connected and simply connected Lie group~$G$ corresponding to~$\mathfrak{g}$ acts upon~$\mathcal{M}$ (see~\cite[Thm.~III, p.~95]{palais}; this follows also from explicitly integrating the action). This action is transitive and, in consequence, $\mathcal{M}$ is a homogeneous space of~$G$.

Let us now describe the structure of~$G$. To~$\mathfrak{h}$ and~$\mathfrak{k}$ correspond two two-dimensional subgroups~$H$ and~$K$ of~$G$. The group~$G$ is the semi-direct product
$H \ltimes K$ corresponding to some representation~$\Psi:H\to\mathrm{Aut}(K)$, which defines, in the set~$H\times K$, the product 
$$(h_1,k_1)\cdot(h_2,k_2)=(h_1+h_2,\Psi_{h_1}(k_2)+k_1).$$
As a transformation group, the element~$(h,k)$ of~$G$ means, with the above conventions for the semi-direct product, to act first by~$h\in H$ by~(\ref{actionh}) and then by~$k\in K$ by~(\ref{actionk}).

Let~$P(\xi)=\xi^2+\alpha \xi+\beta$. Let~$\Delta=\alpha^2-4\beta$ denote its discriminant. We will now give an ordered base~$(Z_1,Z_2)$ to~$\mathfrak{k}$ and calculate~$\Psi$ in each case. In all cases we will have~$[X,Z_i]=-Z_i$.
\begin{itemize}
\item If~$\Delta>0$, $P$ has two different real solutions~$a_1$ and~$a_2$. Set~$Z_i=e^{a_i y}\indel{x}$.  We have~$[Y,Z_i]=a_iZ_i$ and
$$\Psi(s,t )=\left(\begin{array}{rr} e^{s-a_1t} & 0 \\ 0 & e^{s -a_2t}  \end{array}\right).$$
\item If~$\Delta<0$, $P$ has two non-real complex conjugate solutions $a\pm ib$. Up to the rescaling~(\ref{rescaling}), we will suppose that~$b=1$ and set $Z_1=e^{ay}\cos(y)\indel{x}$, $Z_2=e^{ay}\sin(y)\indel{x}$. We have $[Y,Z_1]=aZ_1-Z_2$, $[Y,Z_2]=Z_1+aZ_2$, 
$$\Psi(s,t )=\left(\begin{array}{rr} e^{s-at}\cos(t) & -e^{s-at}\sin(t) \\ e^{s-at}\sin(t) & e^{s-at}\cos(t)  \end{array}\right).$$ 
\item If~$\Delta=0$, $P$ has a real solution~$a$ with multiplicity two. Set $Z_1=e^{ay}\indel{x}$, $Z_2=ye^{ay}\indel{x}$. We have $[Y,Z_1]=aZ_1$, $[Y,Z_2]=aZ_2+Z_1$ and
$$\Psi(s,t )=\left(\begin{array}{rr} e^{s-at} &-te^{s-at} \\ 0 & e^{s -at}  \end{array}\right).$$
\end{itemize}

There is a (non-singular) totally geodesic foliation~$\mathcal{F}_0$ on~$\mathcal{M}$, generated by $\indel{x}$. This foliation may be characterized as being the one tangent to~$[\mathfrak{g},\mathfrak{g}]$. It is thus naturally invariant by~$G$ and, more generally, by the isometry group of~$\nabla_0$ (this accounts for the imprimitiveness of the action of~$G$ on~$M$). By~(\ref{curv4d}), the leaves of this foliation are also the only lies of curvature of the connection.  \\

The group~$G=\mathbf{R}^2\ltimes \mathbf{R}^2$ acts transitively by isometries on~$\mathcal{M}$ and induces the full Killing Lie algebra of~$\nabla_0$. In most cases, it is the group of orientation-preserving isometries:

\begin{prop}\label{isom4d} If~$\mathrm{Isom}^+(\nabla_0)/G$ is not trivial then~$\alpha=2\gamma$ and~$\upsilon=0$; $G$ is the subgroup preserving the orientation of~$\mathcal{F}_0$; $\mathrm{Isom}^+(\nabla_0)/G\approx \mathbf{Z}/2\mathbf{Z}$ and is generated by the class of the orientation-preserving isometry $\sigma(x,y)= (-e^{\alpha y}x,-y)$.   \end{prop}

\begin{proof}  Let~$\phi\in\mathrm{Isom}^+(\nabla_0)$. Up to the action of~$G$ we may suppose that~$\phi$ fixes the origin.  Any isometry should preserve~$\mathcal{F}_0$. Since, in restriction to each leaf, $\indel{x}$ is a geodesic vector field, $\phi$ must act affinely upon the leaves. Similarly, the Killing Lie algebra of~$\nabla_0$ acts upon the leaf space of~$\mathcal{F}_0$ and the image of~$\indel{y}$ generates the image of the Killing algebra in this space. Thus, $\phi$ must act affinely upon the leaf space of~$\mathcal{F}_0$ and must be of the form $$(\overline{x},\overline{y})=\phi(x,y)=(p(y)x+q(y),c^{-1}y),$$
for some~$c\in\mathbf{R}^*$ and some functions~$p$ and~$q$. Under this transformation we have, for every~$h$ such that~$h''+\alpha h'+\beta h=0$,
$$h(y)\del{x}  \mapsto    p(c\overline{y})h(c\overline{y})\del{\overline{x}}.$$
The coefficient~$g(y)=p(cy)h(cy)$  must satisfy the differential equation~(\ref{dekillingk}) with respect to~$y$ and thus, for~$u(\xi)=p'(\xi)/p(\xi)$,
\begin{multline*}\frac{1}{p}(g''+\alpha g'+\beta g)=c\left[2 c u(cy)+\alpha(1-c)\right]h'(cy)+\\+\left[c^2u'(cy)+c^2u^2(cy)+\beta(1-c^2)+\alpha cu(cy)\right]h(cy)=0 \end{multline*}
must vanish term by term. For the first term to vanish we need
\begin{equation}\label{condisom1}u=\frac{p'}{p}=\frac{\alpha}{2}\left(1-\frac{1}{c}\right).\end{equation}
Substituting this into the second term, we obtain
\begin{equation}\label{condisom2}(c-1)(c+1)(\alpha^2-4\beta)=0.\end{equation}
We also have
$$x\del{x} \mapsto  [\overline{x}-q(c\overline{y})]\del{\overline{x}}.$$
In particular, if~$c=1$, from~(\ref{condisom1}), $p$ is a constant (positive by the orientation-preserving hypothesis) and, from the above formula, $q(y)\indel{x}$ must be a Killing vector field. Hence,  $q''+\alpha q'+\beta q=0$. This implies that~$\phi$ belongs to~$G$ and proves the Proposition if~$c\neq 1$.

We will henceforth suppose that~$c\neq 1$. We have
\begin{eqnarray*}
\del{x}  & \mapsto &  A = p(c\overline{y})\del{\overline{x}},\\
\del{y} & \mapsto & B = \left(\frac{p'(c\overline{y})}{p(c\overline{y})}[\overline{x}-q(c\overline{y})]+q'(c\overline{y})\right)\del{\overline{x}}+ c^{-1}\del{\overline{y}}.  
\end{eqnarray*}
In particular, using~(\ref{condisom1}),
\begin{eqnarray*}\nabla_A B-(\gamma+\upsilon)A & = & \frac{1}{2}(\alpha-2\gamma-2\upsilon)\left(1-\frac{1}{c}\right)A\\
\nabla_B A -(\gamma-\upsilon )A & = &\frac{1}{2} (\alpha-2\gamma+2\upsilon)\left(1-\frac{1}{c}\right)A\end{eqnarray*}

Since~$\phi$ is an isometry, the right-hand sides of these equations vanish and thus~$\gamma=\alpha/2$ and~$\upsilon=0$ (in particular, the torsion tensor vanishes). Under these conditions, the curvature~(\ref{curv4d}) will vanish if~$\alpha^2=4\beta$. Since we are assuming that the torsion and the curvature do not vanish simultaneously, from~(\ref{condisom2}), we should have~$c=-1$. If~$\gamma=\alpha/2$ and~$\upsilon=0$, $\sigma$ is an orientation-preserving isometry that belongs to the case~$c=-1$. Up to composing~$\phi$ with~$\sigma$, we are back in the case~$c=1$, proving the Proposition.
\end{proof}

\subsubsection{Global models} We will begin the proof of Theorem~\ref{thmmain}. Let~$S$ be a compact orientable surface endowed with a connection~$\nabla$ such that, for every~$p\in S$, $\nabla$ is locally given by~(\ref{conn4d}), for fixed~$\alpha,\beta,\gamma$ and~$\upsilon$. We begin by proving the following:
\begin{prop}\label{2vf4d} Up to a double covering, $S$ is a torus and there exist two globally defined commuting Killing vector fields of~$\nabla$ in~$S$ that are linearly independent almost everywhere.
\end{prop}

By Propositions~\ref{constgx} and~\ref{isom4d}, up to a double covering,  $S$ has a $(G,\mathcal{M})$\nobreakdash-structure. We have a developing map~$\mathcal{D}:\widetilde{S}\to\mathbf{R}^2$ and a holonomy morphism~$\rho:\pi_1(S)\to G$ that satisfy, for every~$\nu\in\pi_1(S)$ and every~$p\in \widetilde{S}$, the relation~(\ref{dev-hol}). Our aim is to show that the Lie algebra of the Zariski closure of the image of the holonomy provides the required vector fields.

By invariance, $\mathcal{F}_0$ induces a non-singular and oriented foliation~$\mathcal{F}$ in~$S$. By the Poincar\'e-Hopf index Theorem, $S$ is a torus. The fundamental group is a free Abelian group in two generators.

We will begin by investigating the morphisms~$\rho:\mathbf{Z}^2\to G$ (the ones giving all the admissible holonomies) and the corresponding compatible developing maps. In order to prove Proposition~\ref{2vf4d}, we will show that the Zariski closure of~$\rho(\pi_1(S))$, which is naturally invariant by the adjoint action of~$\rho(\pi_1(S))$, acts transitively on some subset of~$\mathcal{M}$.\\

If~$\rho$ takes values in~$H$, the vector fields~$X$ and~$Y$ are preserved under the action of~$\rho$ and are thus well-defined in~$S$. In this case, these are the ones proving Proposition~\ref{2vf4d}.  

If~$\rho$ takes values in the group~$I_{\mathcal{F}_0}$ generated by~$K$ and~$X$ (the Lie algebra of Killing vector fields preserving~$\mathcal{F}_0$ leafwise), the holonomy preserves the function~$y:\mathcal{M}\to \mathbf{R}$. There is thus a function~$\overline{y}:S\to\mathbf{R}$ (induced by~$y$) which must have a maximum. But this is impossible, for~$\overline{y}$ is locally modeled on the submersion~$y$. We conclude that  such morphisms are not holonomy maps, for there is no compatible developing map.

The condition for the elements~$(h_1,k_1)$ and~$(h_2,k_2)$ of~$G$ to commute is that
\begin{equation}\label{comb}(\mathbf{1}-\Psi_{h_1})k_2=(\mathbf{1}-\Psi_{h_2})k_1,\end{equation}
with~$\mathbf{1}$ representing the identity map. We have 
\begin{equation}\label{conjug} (0,q)(h_1,k_1)(0,-q)=(h_1,(\mathbf{1}-\Psi_{h_1})(q)+k_1).\end{equation}

Hence, if~$(\mathbf{1}-\Psi_{h_1})$ is invertible, there exists~$q$ such that~$(0,q)$ conjugates $(h_1,k_1)$ to~$(h_1,0)$. But if~$(h_1,0)$ and~$(h_2,k_2)$ commute and~$(\mathbf{1}-\Psi_{h_1})$ is invertible, from~(\ref{comb}), $k_2=0$. Thus, up to conjugacy, if~$(\mathbf{1}-\Psi_{h_i})$ is invertible for some~$i$, $\rho$ takes values in~$H$. 

There remain the cases where neither~$\mathbf{1}-\Psi_{h_1}$ nor~$\mathbf{1}-\Psi_{h_2}$ are invertible and where the image of the holonomy is not in~$I_{\mathcal{F}_0}$. We will deal with them in function of the structure of~$G$, which is determinmed by~$\Delta$, the discriminant of~$P$.

\subsubsection*{When $\Delta>0$} We will suppose that~$h_i=(s_i,t_i)$, $k_i=(u_i,v_i)$. The lack of invertibility of~$\mathbf{1}-\Psi_{h_i}$ is equivalent to the fact that~$(s_i-a_1t_i)(s_i-a_2t_i)=0$. Up to the natural symmetries, we may suppose that~$s_1=a_1t_1$ with~$t_1\neq 0$ (for if~$h_1$ and~$h_2$ are zero, the holonomy takes values on~$I_{\mathcal{F}_0}$). Since~$s_1\neq a_2t_1$, up to a conjugacy of the form~(\ref{conjug}), we can suppose that~$v_1=0$.

\begin{enumerate}
\item If~$s_2=a_1t_2$, condition~(\ref{comb}) is
$$\left(\begin{array}{rr} 0 & 0 \\ 0 & 1-e^{(a_1 -a_2)t_1}  \end{array}\right)\left(\begin{array}{r} u_2 \\ v_2  \end{array}\right)=\left(\begin{array}{rr} 0 & 0 \\ 0 & 1-e^{(a_1 -a_2)t_2} \end{array}\right)\left(\begin{array}{r} u_1\\ 0  \end{array}\right),$$
this is, $[1-e^{(a_1 -a_2)t_1}]v_2 = 0$. We conclude that~$v_2=0$. There is a Lie group morphism~$\Phi:\mathbf{R}^2\to G=H\ltimes K$ given by
\begin{equation}\label{action1}\Phi(\xi,\zeta)=((a_1\xi,\xi),(\zeta,0)),\end{equation}
corresponding to the Abelian two-dimensional subalgebra of~$\mathfrak{g}$ generated by~$a_1X+Y$ and~$Z_1$. We have that~$(h_i,k_i)=\Phi(t_i,u_i)$. The corresponding action of~$\mathbf{R}^2$ upon~$\mathcal{M}$ is free and transitive.
\item If~$s_2=a_2t_2$, condition~(\ref{comb}) reads
$$\left(\begin{array}{rr} 0 & 0 \\ 0 & 1-e^{(a_1 -a_2)t_1}  \end{array}\right)\left(\begin{array}{r} u_2 \\ v_2  \end{array}\right)=\left(\begin{array}{rr} 1-e^{(a_2 -a_1)t_2} & 0 \\ 0 & 0\end{array}\right)\left(\begin{array}{r} u_1\\ 0  \end{array}\right),$$
this is, $[1-e^{(a_1 -a_2)t_1}]v_2=0$ (hence~$v_2=0$) and~$[1-e^{(a_2 -a_1)t_2}]u_1=0$. If~$t_2=0$, $(h_i,k_i)=\Phi(t_i,u_i)$ for the morphism~(\ref{action1}). If~$t_2\neq 0$, $u_1=0$. This implies that~$(h_1,k_1)$ is the flow of~$a_1X+Y$ in time~$t_1$ and that, for~$c=u_2(a_1-a_2)/(1-e^{(a_2-a_1)t_2})$, $(h_2,k_2)$ is the flow of~$a_2X+Y+cZ_1$ in time~$t_2$. These two vector fields commute and are linearly independent in the complement of the curve~$C_0$ given by $(a_1-a_2)x=ce^{a_1y}$.
\end{enumerate}

\subsubsection*{When $\Delta<0$} The lack of invertibility of~$\mathbf{1}-\Psi_{h_i}$ is equivalent to the fact that
$h_i=(2a m_i\pi,2 m_i\pi)$, for some~$m_i\in\mathbf{Z}$. Up to a finite cover of~$S$ and a change of generators of~$\pi_1(S)$, we may suppose that~$h_1=(2a m \pi,2m\pi)$ and~$h_2=(0,0)$, with~$m\in\mathbf{Z}$, $m\neq 0$. The corresponding actions  upon~$\mathcal{M}$ are respectively given by
$$(x,y)\mapsto (e^{2a m \pi}[x+e^{ay}(\cos(y)u_1+\sin(y)v_1)],y+2m\pi),$$
$$(x,y)\mapsto (x+e^{ay}(\cos(y)u_2+\sin(y)v_2),y).$$
The function~$\overline{y}:\mathcal{M}\to\mathbf{R}/2m\pi\mathbf{Z}$ given by the class of~$y$ is invariant by holonomy. It induces a function~$\overline{y}:S\to\mathbf{R}/2m\pi\mathbf{Z}$ whose level curves are union of leaves of~$\mathcal{F}$ (in particular, the leaves of this foliation are closed). Since~$m\neq 0$ and~$S$ is connected, $y(\mathcal{D}(\widetilde{S}))=\mathbf{R}$. Let~$y_0$ be such that~$\cos(y_0)u_2+\sin(y_0)v_2=0$. Any point of the form~$(x,y_0)$ is fixed by~$(h_2,k_2)$. Let~$L_0$ be the leaf of~$\mathcal{F}_0$ such that~$y(L)=y_0$ and let~$L$ be its image in~$S$. It is a closed geodesic.  Let~$\nu\in\pi_1(S)$ be the element corresponding to~$L$. For~$p\in L_0$ we have the relation~(\ref{dev-hol}). Applying~$y$ to both sides of the equation, we conclude that the action of~$\rho(\nu)$ on~$\mathcal{M}$ does not change the value of~$y$ and thus~$\rho(\nu)$, which belongs to the group generated by~$(h_2,k_2)$, is trivial. However, $\mathcal{D}|_{\pi^{-1}(L)}:\pi^{-
1}(L)\to L_0$ is the developing map of a closed geodesic and its holonomy cannot be trivial. This contradiction shows that there are no compact models in this case.

\subsubsection*{When $\Delta=0$} We have~$h_1=(at_1,t_1)$ with~$t_1\neq 0$ (for if~$h_1$ and~$h_2$ are zero, the holonomy takes values on~$I_{\mathcal{F}_0}$). Condition~(\ref{comb}) reads
$$\left(\begin{array}{rr} 0 & -t_1 \\ 0 & 0  \end{array}\right)\left(\begin{array}{r} u_2 \\ v_2  \end{array}\right)=\left(\begin{array}{rr} 0 & -t_2 \\ 0 & 0  \end{array}\right)\left(\begin{array}{r} u_1\\ v_1  \end{array}\right),$$
this is, $t_1 v_2=t_2 v_1$. For~$c\in\mathbf{R}$, there is a Lie group morphism~$\Phi_c:\mathbf{R}^2\to G$ given by
$\Phi_c(\xi,\zeta)=((a\xi,\xi),(\zeta-\frac{1}{2}c\xi^2,c\xi))$, corresponding to the Lie subalgebra of~$\mathfrak{g}$ generated by~$Z_1$ and  $aX+Y+c Z_2$, whose action upon~$\mathcal{M}$ is free and transitive. We have~$(h_i,k_i)=\Phi_c(t_i,u_i+\frac{1}{2}ct_i^2)$ for~$c=v_1/t_1$. \\

We thus found, in each admissible case, a couple of vector fields in~$\mathcal{M}$ which commute and which  are linearly independent everywhere or in the complement of a curve. They are, by definition, holonomy invariant and are thus induce, via~$\Pi$ and~$\mathcal{D}$, two globally-defined and commuting vector fields in~$S$. This finishes the proof of Proposition~\ref{2vf4d}. \\

Let us now prove Theorem~\ref{thmmain} for the family of connections~(\ref{conn4d}) in the case where the covering of~$S$ produced by Proposition~\ref{constgx} is trivial and, in particular, when $G=\mathrm{Isom}^+(\nabla_0)$. From Proposition~\ref{2vf4d}, we know that there are two commuting, holonomy-invariant, Killing vector fields on~$\mathcal{M}$ which are linearly independent in the complement of a curve~$C_0$. The curve~$\mathcal{D}^{-1}(C_0)$ in~$\widetilde{S}$ is invariant by deck translations and induces a (not necessarily connected) compact curve~$C$ on~$S$ where the vector fields inducing the action of~$\mathbf{R}^2$ have rank smaller than two. Our aim is to prove that~$C$ is empty, this is, that either~$C_0$ is empty or that the image of~$\mathcal{D}$ does not intersect~$C_0$. This would imply that the action of~$\mathbf{R}^2$ in~$S$ is locally free. From the proof of Proposition~\ref{2vf4d}, there are two situations where~$C_0$ is not empty. In both of them, $C_0$ is of the form~$x=f(y)$ for some function~$f$. Hence, every leaf of~$\mathcal{F}_0$ (level curve of~$y$) intersects one and only one point of~$C_0$, and  the corresponding action of~$\mathbf{R}^2$ on~$C_0$ is transitive (the commuting vector fields do not vanish simultaneously at any point). In particular, a linear combination of the commuting vector fields vanishing at some point of~$C_0$ must vanish identically along~$C_0$. Cut~$S$ along~$C$ in order to obtain a finite number of connected manifolds-with-boundary. Let~$\Omega$ be such a manifold. Since~$C$ is transverse to~$\mathcal{F}$, by the Poincar\'e-Hopf index theorem, $\Omega$ is a cylinder and has two boundary components $\mathcal{C}_1$ and~$\mathcal{C}_2$, connected components of~$C$. Let~$V$ be an auxiliary vector field, nowhere vanishing, tangent to the oriented foliation~$\mathcal{F}$ and such that, along~$\mathcal{C}_1$, it points inwards. No curve of~$\mathcal{F}$ starting at $\mathcal{C}_1$ may intersect~$\mathcal{C}_1$ again, for it would do so pointing outwards. No curve of~$\mathcal{F}$ starting 
at $\mathcal{C}_1$ may intersect~$\mathcal{C}_2$, for this would imply that a leaf of~$\mathcal{F}_0$ intersects~$C_0$ twice. Thus, every orbit starting at~$\mathcal{C}_1$ accumulates to some subset of~$\Omega$. By the Poincar\'e-Bendixon theorem,  this limit set 
is a closed orbit~$O$ of~$\mathcal{F}$. Furthermore, $O$ and~$\mathcal{C}_1$ 
are both simple closed curves generating the homology of~$\Omega$. Let~$\gamma$ be a generator of the fundamental group of~$\Omega$. On one hand, since~$\gamma$ is homologous to~$\mathcal{C}_1$, the holonomy of~$\gamma$ must act without fixed points in~$C_0$. On the other, since~$\gamma$ is homologous to~$O$, its holonomy must preserve the leaf~$L$ of~$\mathcal{F}_0$ corresponding to~$O$, and must fix the point~$L\cap C_0$. This contradiction shows that~$C$ is actually
empty. This proves Theorem~\ref{thmmain} in this case.\\

Let us now deal with the case where the covering of~$S$ is not trivial. According to Propositions~\ref{constgx} and~\ref{isom4d}, there is a connected open subset~$\Omega\subset \mathcal{M}$ and a group~$\Gamma\subset \mathrm{Isom}^+(\nabla_0)$ acting properly discontinuously on~$\Omega$ whose quotient is~$S$. There is a subgroup~$\overline{\Gamma}\subset\Gamma$ such that~$[\Gamma:\overline{\Gamma}]=2$ and such that~$\overline{\Gamma}\subset G$. We have~$\overline{S}=\Omega/\overline{\Gamma}$ and a natural two-fold covering~$\pi:\overline{S}\to S$. Let~$\zeta\in\Gamma\setminus G$. The orientation of~$\mathcal{M}$ is preserved by~$\zeta$ but that of~$\mathcal{F}$ is inversed. Because~$\zeta$ is conjugated to the isometry~$\sigma$ of Proposition~\ref{isom4d}, it has a fixed point in~$\mathcal{M}$. The case~$\Omega=\mathcal{M}$ must be excluded, for~$\Gamma$ must act without fixed points in~$\Omega$. In the other cases, $\Omega$ is the complement of the curve~$C_0$ given by $x=f(y)$ for some~$f$. This means that every leaf of~$\mathcal{F}_0$ intersects~$C_0$. But now~$\zeta$ cannot inverse the orientation of~$\mathcal{F}$ and preserve~$\Omega$. We conclude that these cases do not arrive and that the finite covering given by Proposition~\ref{constgx} is in fact trivial. This finishes the proof of Theorem~\ref{thmmain} for the family of connections~(\ref{conn4d}).

\subsection{The homogeneous spaces of~$\mathrm{SL}(2,\mathbf{R})$}
In this case, the model space is a two-dimensional homogeneous space of the universal covering of~$\mathrm{SL}(2,\mathbf{R})$, that we will denote by~$G$. The space is of the form~$G/H$ for~$H$ a closed one-parameter subgroup of~$G$. The connection on~$G/H$ is one that is invariant under the action of~$G$. By Proposition~\ref{constgx}, up to a finite covering, there exists a left-invariant connection~$\nabla_0$ on~$G/H$ and a~$(G,G/H)$\nobreakdash-structure on~$S$ which is an isometry between~$\nabla$ and~$\nabla_0$. This, in its turn, induces a developing map and  holonomy~$\mathcal{D}:\widetilde{S}\to G/H$ and~$\rho:\pi_1(S)\to G$. The center of~$G$ is cyclic and acts upon~$G/H$ by isometries. For any subgroup~$Z$ of the center, there are coverings~$G\to Z\backslash G$ and $G/H\to Z\backslash G/H$. From the above developing map and holonomy we may define naturally~$\mathcal{D}_Z:\widetilde{S}\to Z\backslash G/H$ as~$Z\backslash \mathcal{D}$ and~$\rho_Z:\pi_1(S)\to Z\backslash G$ as~$Z\backslash \rho$, that 
still satisfy the relation~(\ref{dev-hol}). In other words, we may suppose that~$S$ is locally modeled in~$Z\backslash G/H$ with changes of coordinates in~$Z\backslash G$.

The group $\mathrm{SL}(2,\mathbf{R})$ has, up to conjugacy, three one-parameter subgroups:
\begin{equation}\label{sl2sbg}\left\{\left(\begin{array}{cc}e^t & 0  \\ 0 & e^{-t}\end{array}\right)\right\}, \left\{\left(\begin{array}{rr}\cos(t) & \sin(t)  \\ -\sin(t) & \cos(t) \end{array}\right)\right\}, \left\{\left(\begin{array}{cc}1  & t  \\ 0 & 1 \end{array}\right)\right\},\end{equation}
that turn out to be closed. They are called, respectively, \emph{semisimple}, \emph{orthogonal} and \emph{unipotent}. Each one of these gives one two-dimensional homogeneous space of~$\mathrm{SL}(2,\mathbf{R})$.

\subsubsection{Unipotent stabilizer} The associated homogeneous space is~$\mathbf{R}^2\setminus\{0\}$ with the standard linear action of~$\mathrm{SL}(2,\mathbf{R})$. This action preserves the foliation~$\mathcal{F}_0$ of lines through the origin. The surface~$S$ is thus a torus. We have a developing map~$\mathcal{D}:\widetilde{S}\to \mathbf{R}^2\setminus\{0\}$ and a holonomy~$\rho:\pi_1(S)\to\mathrm{SL}(2,\mathbf{R})$.   Since the holonomy is Abelian, it is virtually contained in a one-parameter subgroup and thus, up to a finite covering of~$S$, it is contained in a one-parameter subgroup. Up to conjugacy, we may suppose that the holonomy is contained in one of the groups~(\ref{sl2sbg}). Let us show, in each of these cases, that there exists a holonomy-invariant submersion~$f_0:\mathbf{R}^2\setminus\{0\}\to\mathbf{R}$. In the case of semisimple holonomy, $f_0(x,y)=xy$;  in the case of orthogonal holonomy, $f_0(x,y)=x^2+y^2$; in the case of unipotent holonomy, $f_0(x,y)=x$. This induces a submersion~$f:S\to\mathbf{R}$. However, since~$S$ is compact, there is no submersion taking values in~$\mathbf{R}$ which is globally defined in~$S$.

\subsubsection{Semisimple stabilizer} The case where the model space~$\mathcal{M}$ can be chosen to be the complement of the diagonal in~$\mathbf{RP}^1\times\mathbf{RP}^1$, with the diagonal action of~$\mathrm{PSL}(2,\mathbf{R})$ by fractional linear transformations in the chart~$(x,y)\mapsto([x:1],[y:1])$. The projection onto the first factor gives a foliation~$\mathcal{F}_0$ on~$\mathcal{M}$ and thus~$S$ is a torus, the holonomy is Abelian and is virtually contained in a one-parameter group. The idea is again to produce a submersion from~$S$ to some non-compact curve. We have the following cases: 
\begin{description}
\item[Semisimple holonomy] up to conjugacy, the holonomy belongs to the group of transformations~$(x,y)\mapsto(\lambda x,\lambda y)$ and the function~$f_0:\mathcal{M}\to\mathbf{RP}^1\setminus\{1\}$ given by~$x/y$ is invariant. The level curves of this function give rise to a holonomy-invariant singular foliation~$\mathcal{G}_0$ in~$\mathcal{M}$ and hence one in~$S$. In~$\mathcal{M}$, $\mathcal{G}_0$ has singularities of Poincar\'e-Hopf index~$-1$ at~$(0,\infty)$ and~$(\infty,0)$ and is free from singularities elsewhere. Hence, these points are unattained by~$\mathcal{D}$ and~$f_0$  induces a submersion~$f:S\to \mathbf{RP}^1\setminus\{1\}$. 
\item[Orthogonal holonomy] The function~$f_0:\mathcal{M}\to\mathbf{R}$ given by~$f_0(x,y)=(1+xy)/(x-y)$  is a holonomy-invariant submersion.
\item[Unipotent holonomy] Up to conjugacy, the holonomy is~$(x,y)\mapsto (x+t,y+t)$ and the submersion~$f_0:\mathcal{M}\to\mathbf{R}$, $f_0(x,y)=(x-y)^{-1}$ is invariant. \end{description}
In all cases, $f_0$ induces a submersion~$f$ from~$S$ onto a non-compact curve.

\subsubsection{Orthogonal stabilizer}\label{ossl2} The associated model space is the half-plane~$\mathbf{H}=\{z\in\mathbf{C},\;\Im(z)>0\}$ with the action of~$\mathrm{PSL}(2,\mathbf{R})$ by fractional linear transformations. This action preserves the hyperbolic metric~$g$ together with its Levi-Civita connection~$\nabla_g$. If~$\nabla_0$ is different from~$\nabla_g$, the tensor~$\nabla_g-\nabla_0$ is not identically zero and its image generates a foliation in~$\mathbf{H}$ invariant by the action of~$\mathrm{PSL}(2,\mathbf{R})$. But this action is primitive and there is no preserved foliation: we must conclude that~$\nabla_0=\nabla_g$. 

We claim that the groups isometries of~$g$ and~$\nabla_g$ coincide. From Theorem~\ref{liecorr}, if the Killing Lie algebra of a locally homogeneous connection contains~$\mathfrak{sl}(2,\mathbf{R})$ as a proper subalgebra, this connection is flat. Hence, the Killing Lie algebras of~$g$ and~$\nabla_g$ coincide. Let~$\phi$ be an isometry of~$\nabla_g$. Up to an  isometry of~$g$ we may assume that it is orientation-preserving and that it fixes the point~$i\in\mathbf{H}$. Let~$R$ be a Killing vector field that vanishes at~$i$. The isometry~$\phi$ acts upon the Killing Lie algebra and induces an automorphism of~$\mathfrak{sl}(2,\mathbf{R})$. It fixes the element corresponding to~$R$. But any automorphism of~$\mathfrak{sl}(2,\mathbf{R})$ fixing this element is an inner automorphism corresponding to the adjoint action of~$\exp(tR)$. Hence, up to a rotation fixing~$i$, $\phi$ fixes all the vector fields in the Killing algebra and is thus the identity in~$\mathbf{H}$. The group generated by the Killing Lie algebra 
of~$\nabla_g$ is the group of orientation-preserving isometries of~$g$.

We have a developing map~$\mathcal{D}:\widetilde{S}\to \mathbf{H}$ and a holonomy morphism~$\rho:\pi_1(S)\to \mathrm{Isom}^+(g)$: the previous discussion enables us to do this directly in~$S$ and not in a finite covering. We conclude that~$\nabla$ is the Levi-Civita connection of a globally-defined hyperbolic metric on~$S$. This proves Theorem~\ref{thmmain} in this case.

\subsection{The homogeneous spaces of the orthogonal group~$\mathrm{SO}(3,\mathbf{R})$} The simply connected group with Lie algebra $\mathfrak{o}(3,\mathbf{R})$ is~$\mathrm{U}(2)$. All of its one-parameter subgroups are conjugate and contain the center, which is of order two. Hence, in this case, the model space is the sphere~$S^2$ under the action of~$\mathrm{SO}(3,\mathbf{R})$. The action preserves the Riemannian (round) metric~$g$ together with its Levi-Civita connection~$\nabla_g$. By essentially the same arguments as in~\S \ref{ossl2}, $\nabla_0=\nabla_g$ and the isometry groups of~$g$ and~$\nabla_g$ coincide. As before, we conclude that~$\nabla$ is the Levi-Civita connection of a globally-defined metric of constant positive curvature on~$S$. This proves Theorem~\ref{thmmain} in this case.

\subsection{Two-dimensional algebras} The only two-dimensional Lie algebras (up to isomorphism) are the Abelian one and the one corresponding to the affine group.
\subsubsection{Affine} The \emph{affine group} $\mathrm{Aff}^+(\mathbf{R})$ is the connected and simply connected Lie group $$\{(a,b);\; (a,b)\in\mathbf{R}^2,a > 0\}$$ with the product
$(a_1,b_1)\cdot (a_2,b_2)=(a_1a_2,a_2b_2+b_1)$. We will now deal with the case where~$S$ has a connection~$\nabla$ such that at some point~$p$, the Lie algebra of Killing vector fields has rank two and is isomorphic to the Lie algebra of~$\mathrm{Aff}^+(\mathbf{R})$.  By Proposition~\ref{constgx}, up to replacing~$S$ by a finite covering, there is a left-invariant connection~$\nabla_0$ on~$\mathrm{Aff}^+(\mathbf{R})$, a developing map~$\mathcal{D}:\widetilde{S}\to \mathrm{Aff}^+(\mathbf{R})$ and a holonomy morphism~$\rho:\pi_1(S)\to \mathrm{Aff}^+(\mathbf{R})$ satisfying~(\ref{dev-hol}) and establishing an isometry between~$\nabla$ and~$\nabla_0$.

The foliation~$\mathcal{F}_0$ on~$\mathrm{Aff}^+(\mathbf{R})$ given by the kernel of~$da$ is invariant by left translations. It induces a non-singular foliation~$\mathcal{F}$ on~$S$, which is thus a torus. In particular, $S$ has an Abelian fundamental group. The elements~$(a_1,b_1)$ and~$(a_2,b_2)$ commute if and only if~$b_2(a_1-1)=b_1(a_2-1)$. In the case where the holonomy contains an element of the form~$(a_1,b_1)$ with~$a_1\neq 1$, this element may be conjugated to~$(a_1,0)$. This last element commutes with~$(a_2,b_2)$ if~$b_2=0$. In this case, up to a conjugacy, the image of the holonomy is contained in the subgroup~$\{(a,0)\}$. The function~$f_0(a,b)=b/a$ is holonomy-invariant. In the other case, the holonomy belongs to the subgroup~$\{(1,b)\}$. The function~$f_0(a,b)=a$ is an invariant one. In both cases, there is a holonomy-invariant submersion~$f_0:\mathrm{Aff}^+(\mathbf{R})\to\mathbf{R}$ which induces a submersion~$f:S\to\mathbf{R}$, which is impossible. This proves Theorem~\ref{thmmain} in this 
case.
\subsubsection{Abelian}  This is the case where, in a neighborhood of a point~$p\in S$, there are two linearly-independent, commuting Killing vector fields of~$\nabla$. By Proposition~\ref{constgx}, there exists a finite Galois covering~$\pi:\overline{S}\to S$, there is a translation-invariant connection~$\nabla_0$ on~$\mathbf{R}^2$ and~$\overline{S}$ is endowed with a~$(\mathbf{R}^2,\mathbf{R}^2)$\nobreakdash-structure that is an isometry between~$\overline{\nabla}$ and~$\nabla_0$. Since the adjoint action of~$\mathbf{R}^2$ is trivial, the vector fields on~$\mathbf{R}^2$ are holonomy-invariant and, in consequence, there are two globally-defined, commuting vector fields on~$\overline{S}$ that are Killing vector fields of~$\overline{\nabla}$ and that are linearly independent at every point.  Thus, $\overline{S}$ is a quotient of~$\mathbf{R}^2$ and~$\overline{\nabla}$ is induced by a translation-invariant connection~$\nabla_0$ in~$\mathbf{R}^2$. The group of isometries of~$\nabla_0$ is contained in the affine 
group of~$\mathbf{R}^2$, for~$\mathrm{GL}(2,\mathbf{R})$ is the automorphism group of the Lie algebra~$\mathbf{R}^2$. Hence, $S$ is a compact quotient of~$\mathbf{R}^2$ by a group~$\Gamma$ of affine transformations acting properly. The quotient of~$\mathbf{R}^2$ under the translational part of~$\Gamma$ is the compact surface~$\overline{S}$. We must conclude that~$\overline{S}=S$, this is, $S$ is a quotient of~$\mathbf{R}^2$ and~$\nabla$ is induced by a translation-invariant connection in~$\mathbf{R}^2$.\\

This finishes the proof of~Theorem~\ref{thmmain}.


\end{document}